\documentclass[11pt]{article}
    
\usepackage{amsmath,amsfonts,amsthm,amscd,amssymb,graphicx}

\usepackage{subfigure}
\usepackage{xcolor}

\numberwithin{equation}{section}
\usepackage{hyperref}

\newtheorem{theorem}{Theorem}[section]

\newtheorem{lemma}[theorem]{Lemma}
\newtheorem{lem}[theorem]{Lemma}
\newtheorem{proposition}[theorem]{Proposition}

\def\eps{\varepsilon }

\newcommand{\RR}{\mathbb{R}}

\newcommand{\C}{\mathbb{C}}

\newcommand{\NN}{{\mathbb N}}
\newcommand{\ZZ}{{\mathbb Z}}
\newcommand{\TT}{{\mathbb T}}
\def\beq{\begin{equation}}
\def\eeq{\end{equation}}
\def\bb1{{1\!\!1}}

\def\cL{\mathcal{L}}

\def\eps{\varepsilon}

\def\Ff{\widehat{f}}

\def\Frho{\widehat{\rho}}
\def\FE{\widehat{E}}
\def\FS{\widehat{S}}

\def\FN{\widehat{N}}

\def\Fmu{\widehat{\mu}}

\def\TG{\widetilde{G}}

\begin{document}

\title{Plasma echoes near stable Penrose data} 

\author{Emmanuel Grenier\footnotemark[1]
  \and Toan T. Nguyen\footnotemark[2]
\and Igor Rodnianski\footnotemark[3]
}

\maketitle
\renewcommand{\thefootnote}{\fnsymbol{footnote}}

\footnotetext[1]{CNRS et \'Ecole Normale Sup\'erieure de Lyon, Equipe Projet Inria NUMED,
 INRIA Rh\^one Alpes, Unit\'e de Math\'ematiques Pures et Appliqu\'ees., 
 UMR 5669, 46, all\'ee d'Italie, 69364 Lyon Cedex 07, France. Email: Emmanuel.Grenier@ens-lyon.fr}

\footnotetext[2]{Penn State University, Department of Mathematics, State College, PA 16803. Email: nguyen@math.psu.edu. TN is a Visiting Fellow at Department of Mathematics, Princeton University, and partly supported by the NSF under grant DMS-1764119, an AMS Centennial fellowship, and a Simons fellowship.}

\footnotetext[3]{Princeton University, Department of Mathematics, Fine Hall, Washington Road, Princeton, NJ 08544. Email: irod@math.princeton.edu. IR is partially supported by the NSF 
grant DMS \#1709270 and a Simons Investigator Award. }

\subsubsection*{Abstract}

In this paper we construct particular solutions to the classical Vlasov-Poisson system near stable Penrose initial data on $\TT \times \RR$ that are a combination of elementary waves with arbitrarily high frequencies. These waves mutually interact giving birth, eventually, to an infinite cascade of echoes of smaller and smaller amplitude. The echo solutions do not belong to the analytic or Gevrey classes studied by Mouhot and Villani, but do, nonetheless, exhibit damping phenomena for large times.  

%\end{abstract}

%%%%%%%%%%%%%%%%%%%%

\section{Introduction}\label{sec-echosolution}

%%%%%%%%%%%%%%%%%%%%

In the physical literature, the large time behavior of a plasma modeled by the classical Vlasov-Poisson system is characterized by 

\begin{itemize}

{\em \item Landau damping:} decay of the electric field for large times. 

{\em \item Plasma echoes.} An elementary wave, arising as a {\it free transport} of initial data of the form $\eps f_1(v) e^{i k_1 x + i \eta_1 v}$, will generate an electric field of order $\epsilon$ that is localized near the critical time $\tau_1 = \eta_1/k_1$ and decays\footnote{polynomially or exponentially small, depending on the regularity of $f_1(v)$.} for larger times. When two elementary waves $\eps f_j(v) e^{i k_j x + i \eta_j v}$, with arbitrarily large frequencies $k_j, \eta_j$, $j = 1,2$ and the associated critical times $\tau_j = \eta_j / k_j$ interact, a third wave of the same form is created. The electric field of this third wave is localized, but of order $\epsilon^2$, near the ``echo time'' $\tau = (\eta_1 + \eta_2) / (k_1 + k_2 )$, which could be long after the first two waves have died away. The phenomenon is often referred to as an ``echo" \cite{Gould}.

\end{itemize}

From the discussion above,  an echo is of a higher order ($\eps^2$) in amplitude. The created wave again interacts with the other two waves,
creating higher order waves, and higher order echoes, and so on. That is, starting from two waves, an infinite number
of waves, of smaller and smaller amplitudes, appear, with an infinite number of echoes, of smaller and smaller
amplitudes.

In this context, the fundamental question is to understand the described heuristic picture and analyze  large time {\it nonlinear} behavior
of  such an infinite cascade. While the linear Landau damping was discovered and fully understood by Landau \cite{Landau}, 
the nonlinear analogue has been largely elusive. However, important progress has been made by Mouhot and Villani in their celebrated work \cite{MV}, where the problem was solved in the case of analytic or Gevrey data. 
Their proof has then been simplified in \cite{BMM}. The damping for data with finite Sobolev regularity remains largely open due to plasma echoes \cite{Gould} and high frequency instabilities \cite{Bedrossian}, while it is known to be false for data with very low regularity \cite{Lin}.  

In the companion paper \cite{GNR1}, we give an elementary proof of the nonlinear Landau damping for analytic and Gevrey data \cite{MV,BMM}. In this paper, we construct an exact solution of the classical Vlasov-Poisson system, starting form an infinite number of elementary waves
of amplitude $\eps$. Provided that $\epsilon$ is sufficiently small and each wave has an analytic regularity, we are able to track all the interactions and to construct solutions which display an infinite number of
echoes, which are of a smaller and smaller amplitude as time evolves. The associated electric field decays for large times, and therefore {\em Landau damping} holds for such data. As we are allowed to take the frequencies of each elementary wave to be arbitrarily large, the solutions do not belong to the class of analytic or Gevrey solutions constructed in \cite{MV,BMM}.

Precisely, we consider the following classical Vlasov-Poisson system\footnote{obtained from the standard Vlasov-Poisson equations
$$
\partial_t \tilde f + v\partial_x \tilde f +E \partial_v \tilde f  =0, \qquad \partial_x E = \rho = \int_{\RR} \tilde f\; dv-1
$$
by writing $\tilde f=f+\mu$} 
\begin{equation}\label{VP-pert} 
\partial_t f + v\partial_x f +E \partial_v f + E \partial_v \mu =0, \qquad \partial_x E = \rho = \int_{\RR} f\; dv
\end{equation}
on the torus $\TT \times \RR$, for small initial data $f^0(x,v)$, where $\mu(v)$ is a stable Penrose equilibrium. We require that $\mu(v)$ is such that: 
\begin{itemize}

\item $\mu(v)$ is real analytic and satisfies 
\begin{equation}
\label{reg-mu}
|\widehat{\langle v\rangle^2\mu}(\eta)| + |\widehat{\mu}(\eta)|  \le C_0 e^{-\theta_0 |\eta|}
\end{equation} 

\item $\mu(v)$ satisfies the Penrose stability condition: namely, 
\begin{equation}\label{Penrose0}
\inf_{k\in \ZZ^d  \setminus\{0\}; \Re \lambda \ge 0} \Big |1 + \int_0^\infty e^{-\lambda t} t \widehat{\mu}(kt) \; dt \Big| \ge \kappa_0 >0.
\end{equation}
The condition is to ensure that the linearized system of \eqref{VP-pert} (obtained by dropping the nonlinear term $E \partial_v f$) is solvable. It holds for a variety of equilibria including the Gaussian $\mu(v) = e^{-|v|^2/2}$. 
In three or higher dimensions, the condition is valid for any positive and radially symmetric equilibria \cite{MV}.

\end{itemize}

We consider the initial data which are a sum of highly oscillatory simple modes of the form
\beq \label{echoinit}
f^0(x,v) = \sum_{(k,\eta) \in \ZZ\setminus\{0\} \times \ZZ}  \eps f^0_{k,\eta}(v) e^{i K kx + i L\eta  v} ,
\eeq
for large parameters $K$ and $L$ and for small $\epsilon$. We stress that $f^0$ is rapidly oscillating in $x$ and $v$. In particular, of $K,L$ are allowed to be arbitrarily large,
the initial data is of order $\epsilon \langle K,L\rangle^s$ in Sobolev spaces $W^{s,\infty}(\TT \times \RR)$,
 which is also arbitrarily large, for any $s>0$. 

Our main result asserts that Landau damping holds for data of the form \eqref{echoinit}. Precisely, we prove the following.

\begin{theorem}\label{theo-echoes} Let $\lambda_0, C_0>0$, and $K,L$ be arbitrarily large so that 
$$ L \le C_0 K. $$
Assume that \eqref{reg-mu}-\eqref{Penrose0} hold. Then, for sufficiently small $\epsilon$ independent of $K,L$ 
and for any initial data of the form \eqref{echoinit} with the analytic coefficients $f^0_{k,\eta}(v)$ satisfying\footnote{We use the notation $\langle x_1,..,x_n\rangle=\sqrt{1+x_1^2+...+x_n^2}$. Depending on the context, we also use $~\widehat{} ~$ to denote the Fourier transform in $x$, $v$ or both.}
\begin{equation}\label{assump-p1}
|\Ff^0_{k,\eta}(\eta')|\le e^{-2\lambda_0 \langle k, \eta,\eta'\rangle},
\end{equation}
uniformly in $k,\eta,\eta'$,
 there exists a unique global solution to the Vlasov-Poisson system \eqref{VP-pert}. In addition, the solution can be written in the form 
\beq \label{formal1}
f(t,x,v) = \sum_{(k,\eta, p) \in \ZZ \times \ZZ \times \NN^\star} \eps^p f_{k,\eta,p}(t,v) e^{i K kx + i (L\eta - Kkt) v} ,
\eeq
where the coefficients $f_{k,\eta,p}(t,v)$ are analytic in $v$ and satisfy 
\begin{equation}\label{main-bdF}|\Ff_{k,\eta,p} (t,\eta')| \le C_1^p  e^{- \lambda_0\langle k,\eta,p, \eta' \rangle} , \quad ~t\ge 0,
\end{equation}
uniformly in $k,\eta,p,\eta'$, for some universal constant $C_1$ that is independent of $K,L$, and $\epsilon$. In particular, the associated electric field 
$$E(t,x) = \sum_{(k,\eta, p) \in \ZZ \times \ZZ \times \NN^\star} \eps^p \Ff_{k,\eta,p}(t, Kkt-L\eta ) \frac{e^{i K kx }}{iK k} \longrightarrow 0$$
exponentially fast in any Sobolev spaces $W^{s,q}$, $s\ge 0$ and $q\ge 1$, as $t\to \infty$. 
\end{theorem}
We should note that the results easily generalize to higher dimensions. Condition $L\le C_0 K$, coupled with the assumptions on the $(k,\eta)$ dependence of the initial data, imply that the ``echos" occur at times which are, essentially, 
uniformly bounded.

%%%%%%%%%%%%%%%%%%%%%

\section{Linear theory}

%%%%%%%%%%%%%%%%%%%%%

In this section, we recall the linear Landau damping theory developed in the companion paper \cite{GNR1}. Precisely, let $k\in \ZZ$ and let $\Frho(t)$ satisfy 
\begin{equation}
\label{eqs-rho0} \Frho(t) + \int_0^t (t-s) \widehat{\mu}(k(t-s)) \Frho(s) \; ds = \FS(t)
\end{equation}
with a source term $\FS(t)$. Taking the Laplace transform of \eqref{eqs-rho0} in time, we  get 
\begin{equation}\label{Lap-rho} 
\cL[\Frho(t)](\lambda) = \frac{\cL[\FS(t)](\lambda)}{1 + \cL[t \widehat{\mu}(kt)](\lambda)} 
\end{equation}
where $\cL[F(t)](\lambda)$ denotes the usual Laplace transform of $F(t)$. The Penrose stability condition \eqref{Penrose0} ensures that the symbol $1 + \cL[t \widehat{\mu}(kt)](\lambda)$ never vanishes. 

We then have the following. 

\begin{proposition}\label{prop-GNR} Assume that \eqref{reg-mu}-\eqref{Penrose0} hold. Then, the solution $\Frho(t)$ to \eqref{eqs-rho0} exists and satisfies 
\begin{equation}\label{intform-rho}
\Frho(t)  = \FS(t) + \int_0^t G_k(t-s) \FS(s)\; ds
\end{equation}
where $|G_k(t)| \le C_1 e^{-\theta_1 |kt|}$ for some positive constants $\theta_1,C_1$. 

\end{proposition}
\begin{proof} From \eqref{Lap-rho}, we can write 
\begin{equation}
\label{eqs-rholambda}\cL[\Frho](\lambda) = \cL[\FS](\lambda) +  \TG_k(\lambda) \cL[\FS](\lambda)
\end{equation} 
where we denote 
\begin{equation}\label{def-Glambda}
\TG_k(\lambda):= - \frac{\cL[t \widehat{\mu}(kt)](\lambda)}{1 +\cL[t \widehat{\mu}(kt)](\lambda)} .
\end{equation}
The integral formulation \eqref{intform-rho} thus follows, where $G_k(t)$ is the inverse Laplace transform of $\TG_k(\lambda)$.  It remains to prove the estimate on $G_k(t)$. First, we note by definition that 
$$ 
\cL[t \widehat{\mu}(kt)](\lambda) = \int_0^\infty e^{-\lambda t} t \Fmu(k t) \; dt .
$$
Thus, the Penrose condition \eqref{Penrose0} ensures that the denominator $1 +\cL[t \widehat{\mu}(kt)](\lambda)$ never vanishes for $\Re \lambda \ge 0$. Furthermore, using \eqref{reg-mu}, we in fact have 
\begin{equation}\label{bd-Lmu} 
|\cL[t \widehat{\mu}(kt)](\lambda)| \le C_0 \int_0^\infty e^{-\Re \lambda t} t e^{-\theta_0 |kt|} \; dt \le C_1|k|^{-2}
\end{equation}
for $\Re \lambda \ge - \theta_1 |k|$ and for any $\theta_1<\theta_0$. On the other hand, for $\Re \lambda = -\theta_1 |k|$, integrating by parts in time, we get  
$$
\begin{aligned}  
\cL[t \widehat{\mu}(kt)](\lambda) &= \int_0^\infty \frac{(M^2_k-\partial_t^2)(e^{-\lambda t})}{M^2_k-\lambda^2} t \Fmu(k t) \; dt 
\\&= \int_0^\infty \frac{e^{-\lambda t}}{M^2_k-\lambda^2} (M^2_k-\partial_t^2)(t \Fmu(k t)) \; dt - \frac{\Fmu (0)}{M^2_k - \lambda^2} 
\end{aligned}$$
for any constant $M_k \not = \lambda$. Taking $M_k = 2\theta_1|k|$, we have 
$$
\begin{aligned}  
| \cL[t \widehat{\mu}(kt)](\lambda)| &\le C_0\int_0^\infty \frac{e^{\frac12 \theta_0 |kt|}}{\theta_1^2 |k|^2+ |\Im \lambda|^2} (|k| + |k|^2 t) 
e^{-\theta_0 |kt |} \; dt + \frac{C_0}{\theta_1^2|k|^2 + |\Im \lambda|^2}
\end{aligned}$$
which gives \begin{equation} \label{bd-Lmu1}
\begin{aligned}  
| \cL[t \widehat{\mu}(kt)](\lambda)| 
\le C_1 (1 + |k|^2 + |\Im \lambda |^2)^{-1}. 
\end{aligned}\end{equation}
for any $\lambda$ on the line $\{\Re \lambda = -\theta_1 |k|\}$. This proves that $\cL[t \widehat{\mu}(kt)](\lambda)$ is analytic in $\{\Re \lambda \ge -\theta_1 |k|\}$, for any $\theta_1<\theta_0$.  

We next prove that there is a positive $\theta_1<\theta_0$ so that $ \TG_k(\lambda)$ is analytic in $\{\Re \lambda \ge -\theta_1 |k|\}$ and the estimate \eqref{bd-Lmu1} also holds for $ \TG_k(\lambda)$, possibly with a different constant $C_1$. Indeed, the estimate \eqref{bd-Lmu}  shows that there are $k_0, \tau_0$ so that $|\cL[t \widehat{\mu}(kt)](\lambda)| \le \frac12$ 
for all $|k|\ge k_0$ and $\lambda \ge -\theta_1 |k|$, or for $\Re \lambda = -\theta_1|k|$ and $|\Im \lambda| \ge \tau_0$. While for $|\Im\lambda|\le \tau_0$ and $|k|\le k_0$, since the Penrose condition \eqref{Penrose0} holds for $\Re \lambda =0$, there is a small positive constant $\theta_1$ so that 
\begin{equation}\label{lowbd-Lmu}
|1 + \cL[t \widehat{\mu}(kt)](\lambda)| \ge \frac12 \kappa_0 
\end{equation}
for $\Re \lambda = -\theta_1|k|$ (recalling that $1\le |k|\le k_0$). Combining, we have that $1 + \cL[t \widehat{\mu}(kt)](\lambda)$ is bounded below away from zero on $\{ \Re \lambda \ge -\theta_1 |k|$ for all $k\in \ZZ$. The estimates on $\TG_k(\lambda)$ thus follows from those on $\cL[t \widehat{\mu}(kt)](\lambda)$. 

By definition, we have 
$$
G_k(t) = \frac{1}{2\pi i} \int_{\{ \Re \lambda = \gamma_0\}} e^{\lambda t} \TG_k(\lambda)\; d\lambda
$$ 
for some large positive constant $\gamma_0$. Since 
$\TG_k(\lambda)$ is analytic in $\{\Re\lambda \ge -\theta_1|k| \}$, 
and thus we can apply the Cauchy's theory to deform the complex contour of integration from 
$\{ \Re \lambda = \gamma_0\}$ into $\{ \Re \lambda = - \theta_1|k|\}$, 
on which both estimates \eqref{bd-Lmu1} and \eqref{lowbd-Lmu} hold. Therefore,  
$$
|G_k(t)| \le C_1 \int_{\{ \Re \lambda = -\theta_1 |k|\}} e^{-\theta_1|kt|} (1 + |k|^2 + |\Im \lambda|^2)^{-1}\; d\lambda 
\le C_1 e^{-\theta_1|kt|}.$$ 
The Proposition follows. 
\end{proof}

%%%%%%%%%%%%%%%%%%%%%

\section{Construction}

%%%%%%%%%%%%%%%%%%%%%

%%%%%%%%%%%%

\subsection{Setup}

%%%%%%%%%%%%

Let us first detail the construction of the profiles $f_{k,\eta,p}$. At each step we consider the term $E \partial_v f$
of (\ref{VP-pert}) as a source term for the linear Vlasov Poisson near the equilibrium $\mu$.
For each $(k,\eta, p) \in \ZZ \times \ZZ \times \NN^\star$, 
we thus look for  $f_{k,\eta,p}(t,v)$ and $\FE_{k,\eta,p}(t)$ satisfying 
\begin{equation}
\label{inductive-f}
\partial_t f_{k,\eta,p} + \FE_{k,\eta,p} e^{-i(L\eta - Kkt)v}\partial_v \mu = N_{k,\eta,p}, 
\end{equation}
in which
\begin{itemize}
	
\item For $p=1$, we take $f_{k,\eta,1}(0,v) = f^0_{k,\eta}(v)$ and $N_{k,\eta,1} =0$. 

\item For $p\ge 2$, we take $f_{k,\eta,p}(0,v) =0$ and 
$$
\begin{aligned}
N_{k,\eta,p}(t,v) &= \sum_{(k_1,\eta_1,k_2,\eta_2,p_1,p_2) \in A_{k,\eta,p}} e^{-i(L \eta_1 - Kk_1t)v} \FE_{k_1,\eta_1,p_1}(t) 
\\&\quad \times [\partial_v + i(L \eta_2 - Kk_2 t)] f_{k_2,\eta_2,p_2} (t,v)
\end{aligned}$$
where $A_{k,\eta,p}$  denotes the set of sextets in $\ZZ \times \ZZ \times \NN^\star$:
$$
A_{k,\eta,p}: = \Big \{ 
k_1 + k_2 = k, ~\eta_1 + \eta_2 = \eta, ~p_1 + p_2 = p\Big\}.
$$
\end{itemize}
 The electric field is defined by 
$$
\FE_{k,\eta,0}(t) =0
$$ 
and for $p > 0$, by a direct computation,
\begin{equation}\label{def-Ekp}
 \FE_{k,\eta,p}(t) =  \frac{1}{iKk}\Frho_{k,\eta,p}(t) =  \frac{1}{iKk}\Ff_{k,\eta,p}(t, Kkt-L\eta).
 \end{equation}
By construction, the infinite series \eqref{formal1} formally solves the Vlasov-Poisson system \eqref{VP-pert} 
with corresponding electric field 
\beq \label{formal-E}
E(t,x) = \sum_{(k,\eta, p) \in \ZZ \times \ZZ \times \NN^\star} \eps^p \Ff_{k,\eta,p}(t, Kkt-L\eta ) \frac{e^{i K kx }}{iK k} .
\eeq
Note that taking the Fourier transform of \eqref{formal1} in $x$ and $v$, we have 
\beq \label{formal1-Fourier}
\Ff(t,Kk,\eta') = \sum_{(\eta, p) \in \ZZ \times \NN^\star}  \eps^p \Ff_{k,\eta,p}(t,\eta' - L\eta + Kkt )
\eeq
and $\Ff(t,k',\eta') =0$ for $k' \not \in K \ZZ$. 

It remains to derive estimates on the Fourier transform $\Ff_{k,\eta,p}(t,\eta')$ 
of each functions $f_{k,\eta,p}(t,v)$ in order to ensure the convergence of the infinite series \eqref{formal1-Fourier}.

%%%%%%%%%%%%

\subsection{Resolution using Penrose's kernel}

%%%%%%%%%%%%

We begin by converting (\ref{inductive-f}) to an integral equation.

\begin{lemma}\label{lem-linearFkp} Let $f_{k,\eta,p}(t,v)$ be constructed as indicated above. Set 
\begin{equation}
\label{def-FSkp}
\FS_{k,\eta,p}(t,\eta') = \Ff_{k,\eta,p}(0,\eta' ) + \int_0^t\FN_{k,\eta,p}(s,\eta' ) \; ds .
\end{equation} 
Then, there holds
$$
\Ff_{k,\eta,p}(t,\eta') = \FS_{k,\eta,p}(t,\eta') - \int_0^t \FE_{k,\eta,p}(s) \widehat{\partial_v \mu}(\eta'+ L\eta -Kks) \; ds.  
$$
In addition, 
\begin{equation}
\label{bd-rhoketap}
\Frho_{k,\eta,p}(t)  = \FS_{k,\eta,p}(t,Kkt-L\eta ) + \int_0^t G_k(t-s) \FS_{k,\eta,p}(s,Kks-L\eta )\; ds
\end{equation}
where $|G_k(t)| \le C_0 e^{-\theta_0 |Kkt|}$. 

\end{lemma}

\begin{proof} Integrating \eqref{inductive-f} in time, we obtain 
$$
f_{k,\eta,p}(t,v) = - \int_0^t \FE_{k,\eta,p}(s) e^{-i (L\eta  - Kks)v} \partial_v \mu(v) \; ds + S_{k,\eta,p}
$$
where
$$ 
S_{k,\eta,p}(t,v) = f_{k,\eta,p}(0,v) + \int_0^t  N_{k,\eta,p}(s,v) \; ds .
$$
Taking the Fourier transform yields the expression for $\Ff_{k,\eta,p}(t,\eta')$. 
In particular, using \eqref{def-Ekp}, we have 
$$
\Frho_{k,\eta,p}(t) + \int_0^t  (t-s)\widehat{\mu}(Kk(t-s)) \Frho_{k,\eta,p}(s)\; ds = \FS_{k,\eta,p}(t,Kkt-L\eta ).$$
Using the linear theory developed in Proposition \ref{prop-GNR}, the lemma follows. 
\end{proof}

%%%%%%%%%%%%%%%%%%%%%%

\subsection{Inductive estimates}

%%%%%%%%%%%%%%%%%%%%%%

In this section, we shall inductively derive estimates on $\Ff_{k,\eta,p} (t,\eta')$. 
In what follows, we fix $\lambda_0>0$ and $K,L$ to be arbitrarily large so that
\begin{equation}\label{bounded-criticaltime} 
L \lesssim K.
\end{equation}
Then, we have the following. 

\begin{proposition}\label{prop-fkep}
There is some universal constant $C_0$ so that 
\begin{equation}\label{est-Fp}
|\Ff_{k,\eta,p} (t,\eta')| \le C_0^p  e^{-\lambda_p(t) \langle k,\eta,p, \eta' \rangle} \langle k\rangle^{-1},
\eeq
\beq \label{est-Ep}
|\Frho_{k,\eta,p }(t)| \le C_0^{p} e^{-\lambda_{p}(t) \langle  k,\eta,p,L\eta - K k t\rangle} \langle t \rangle^{-\sigma},
\end{equation}
uniformly in $k$, $\eta$, $p$, $\eta'$ and $t\ge 0$, where $\lambda_p(t)$ is defined by 
\begin{equation}\label{def-lambdap}
\lambda_p(t) = \lambda_0 + \langle t\rangle^{-\delta} + p^{-\delta},
\end{equation} 
for some $0<\delta\ll1$. 
\end{proposition}

Note that all the estimates are uniform in the large parameters $K$ and $L$. 
The following subsections are devoted to the proof of this Proposition, which will be done by induction on $p \ge 1$.

%%%%%%%%%

\subsection{Estimates for $p=1$}

%%%%%%%%%

We first estimate $\FS_{k,\eta,p}(t,\eta')$ for $p=1$. By construction, $N_{k,\eta,1} =0$, 
and thus we have 
$$
\FS_{k,\eta,1}(t,\eta') = \Ff_{k,\eta,1}(0,\eta' ) = \Ff^0_{k,\eta}(\eta' ).
$$
Thus, using the assumption \eqref{assump-p1} in \eqref{bd-rhoketap}, we obtain 
$$\begin{aligned}
| \Frho_{k,\eta,1}(t)| 
& \le |\FS_{k,\eta,1}(t,Kkt-L\eta )| + \int_0^t |G_k(t-s) \FS_{k,\eta,1}(s,Kks-L\eta )|\; ds
\\
& \le e^{-2\lambda_0 \langle k,\eta,L\eta - Kkt\rangle} + C_0 \int_0^t e^{-\theta_0 |Kk(t-s)|} 
e^{-2\lambda_0 \langle k,\eta,L\eta - Kks\rangle} \; ds .
\end{aligned}$$
Using $\lambda_0 \le \theta_0 /4$ and the triangle inequality, we bound 
$$
e^{-\frac12\theta_0 |Kk(t-s)|} 
e^{-2\lambda_0 |L\eta - Kks|} \le e^{-2\lambda_0 |L\eta - Kkt|} .
$$
Hence, 
$$\begin{aligned}
| \Frho_{k,\eta,1}(t)| 
& \le e^{-2\lambda_0 \langle k,\eta,L\eta - Kkt\rangle} + C_0e^{-2\lambda_0 \langle k,\eta,L\eta - Kkt\rangle}  \int_0^t e^{-\frac12\theta_0 |Kk(t-s)|} 
\; ds .
\\
& \le C_0 e^{-2\lambda_0 \langle k,\eta,L\eta - Kkt\rangle} .
\end{aligned}$$
To complete the proof of \eqref{est-Ep} for $p=1$, we need to check the decay in time. Indeed, using the triangle inequality $$ 
|K kt | \le |K kt - L \eta | + |L \eta| 
$$
and the fact that $K \ge 1$ and $L \lesssim K$, we have 
\begin{equation}
\label{tri-kt}
| kt | \le K^{-1}|K kt - L \eta | + L K^{-1} |\eta| \le |K kt - L \eta | +  |\eta|  .
\end{equation} 
This proves that 
$$\begin{aligned}
| \Frho_{k,\eta,1}(t)| 
& \le C_0 e^{-\lambda_0 \langle k,\eta,L\eta - Kkt\rangle} e^{-\lambda_0 \langle kt\rangle},
\end{aligned}$$
which proves \eqref{est-Ep} for $p=1$, since $k\not =0$. To estimate \eqref{est-Fp}, we use Lemma \ref{lem-linearFkp} to estimate 
$$
\begin{aligned}
|\Ff_{k,\eta,1}(t,\eta')| 
&\le |\FS_{k,\eta,1}(t,\eta')| + \int_0^t |\FE_{k,\eta,1}(s) \widehat{\partial_v \mu}(\eta'+ L\eta - Kks)| \; ds 
\\
&\le e^{-2\lambda_0 \langle k,\eta,\eta' \rangle} \\&\quad  + C_0 \langle Kk\rangle^{-1}
\int_0^t e^{-\lambda_0 \langle k,\eta,L\eta - Kks\rangle} e^{-\theta_0 |\eta'+ L\eta - Kks|} \langle s\rangle^{-\sigma}  \; ds 
\\
&\le e^{-2\lambda_0 \langle k,\eta,\eta' \rangle}  + C_0 \langle Kk\rangle^{-1}e^{-\lambda_0 \langle k,\eta,\eta' \rangle} 
\\
&\le C_0 \langle k\rangle^{-1}e^{-\lambda_0 \langle k,\eta,\eta' \rangle} . 
\end{aligned}$$
where we used the exponential decay of the electric field, proven above, to insert an extra factor $\langle s\rangle^{-\sigma}$
with $\sigma>1$.
This proves Proposition \ref{prop-fkep} for $p=1$.

%%%%%%%%%%

\subsection{Estimates on $\FE_{k,\eta,p}$}

%%%%%%%%%%
 
In this section, we shall prove the estimates \eqref{est-Ep} on $\FE_{k,\eta,p}$ for $>1$, 
under the inductive assumption that the estimates \eqref{est-Fp}-(\ref{est-Ep}) on $\Ff_{k,\eta,p_1}$ and $\FE_{k,\eta,p_1}$ 
hold for all $p_1\le p-1$. Precisely, we prove

\begin{lemma}\label{lem-bdEkp}  
Let $P>1$. Assume that (\ref{est-Fp})-(\ref{est-Ep}) hold true for any $k$, $\eta$ and $p \le P -1$. Then
(\ref{est-Ep}) is true for any $k$, $\eta$ and $p = P$.
\end{lemma}

In view of Lemma \ref{lem-linearFkp}, we first prove the following. 

\begin{lemma}\label{lem-bdSkp} Under the assumption of Lemma \ref{lem-bdEkp}, there holds 
$$|\FS_{k,\eta,p}(t,Kkt-L\eta)| \le C_0^pe^{-\lambda_p(t) \langle k,\eta,p, L\eta -  K kt \rangle}  \langle t\rangle^{-\sigma} $$
where $S_{k,\eta,p}$ is defined as in \eqref{def-FSkp}. 
\end{lemma}

\begin{proof}[Proof of Lemma \ref{lem-bdEkp} using Lemma \ref{lem-bdSkp}]
By Lemma \ref{lem-linearFkp}, we have 
\begin{equation}\label{bd-Frhokp}
\Frho_{k,\eta,p}(t)  = \FS_{k,\eta,p}(t,Kkt-L\eta ) + \int_0^t G_k(t-s) \FS_{k,\eta,p}(s,Kkt-L\eta )\; ds
\end{equation}
where $|G_k(t)| \le C_0 e^{-\theta_0 |Kkt|}$. Using Lemma \ref{lem-bdSkp}, we have 
$$ 
\begin{aligned}
|\Frho_{k,\eta,p}(t)|  &\le  C_0^pe^{-\lambda_p(t) \langle k,\eta,p, L\eta -  K kt \rangle}  \langle t\rangle^{-\sigma} 
\\&\quad + C_0^p\int_0^t e^{-\theta_0 |Kk(t-s)|} e^{-\lambda_p(s) \langle k,\eta,p, L\eta -  K ks \rangle}  \langle s\rangle^{-\sigma}\; ds .
\end{aligned}
$$
Using $\lambda_p(t) \le \lambda_p(s) \le \frac12\theta_0$, we have 
$$
\begin{aligned}
e^{- \frac12 \theta_0 |Kk(t-s)|} e^{-\lambda_p(s) \langle k,\eta,p, L\eta -  K ks \rangle}  
&\le e^{- \lambda_p(t) |Kk(t-s)|} e^{-\lambda_p(t) \langle k,\eta,p, L\eta -  K ks \rangle} 
\\
&\le e^{- \lambda_p(t) \langle k,\eta,p, L\eta -  K kt \rangle} .
\end{aligned}$$
On the other hand, since $k\not =0$, we easily bound
$$
\int_0^t e^{-\frac12 \theta_0 |Kk(t-s)|} \langle s\rangle^{-\sigma}\; ds \le C_0 \langle t\rangle^{-\sigma}. 
$$
The desired estimates on $\Frho_{k,\eta,p}(t)$ follow. 
\end{proof}

\begin{proof}[Proof of Lemma \ref{lem-bdSkp}] By construction, for $p>1$, $\Ff_{k,\eta,p}(0,\eta') =0$, and thus we have 
$$
\FS_{k,\eta,p}(t,\eta') = \int_0^t\FN_{k,\eta,p}(s,\eta' ) \; ds 
$$
where the nonlinear interaction $\FN_{k,\eta,p}(t,\eta')$ is computed by 
$$\FN_{k,\eta,p}(t,\eta') = i\sum_{A_{k,\eta,p}} \FE_{k_1,\eta_1,p_1}(t) [\eta' + L \eta - Kkt] \Ff_{k_2,\eta_2,p_2} (t,\eta' + L\eta_1 - Kk_1t).$$
By induction, for $p_1,p_2 \le p-1$, we have 
$$
\begin{aligned} 
|\FE_{k_1,\eta_1,p_1 }(t)| &\le C_0^{p_1}
 e^{-\lambda_{p_1}(t) \langle k_1,\eta_1,p_1, L\eta_1 - K k_1 t \rangle} |Kk_1|^{-1}\langle t \rangle^{-\sigma} 
\\
|\Ff_{k_2,\eta_2,p_2}(t, \eta')| &\le C_0^{p_2}
e^{-\lambda_{p_2}(t) \langle k_2,\eta_2,p_2, \eta' \rangle} \langle k_2\rangle^{-1}.
\end{aligned}$$
Hence, recalling the definition of $A_{k,\eta,p}$, we have 
$$
\begin{aligned}
|\FS_{k,\eta,p}(t,\eta')| &\le C_0^p\sum_{A_{k,\eta,p}} |Kk_1|^{-1}  \langle k_2\rangle^{-1} \int_0^t e^{-\lambda_{p_1}(s) \langle k_1,\eta_1,p_1, L\eta_1 - K k_1 s \rangle} \langle s \rangle^{-\sigma} 
\\&\quad \times |\eta' + L\eta - Kks| e^{-\lambda_{p_2}(s) \langle k_2,\eta_2,p_2, \eta' + L\eta_1 - Kk_1 s\rangle} \; ds. 
\end{aligned}$$
It is crucial to note that $\lambda_p(t)$ is strictly decreasing in both $p$ and $t$.  We will
use this monotonicity in order to gain time decay in the estimates.
Using $k = k_1 + k_2$, $\eta = \eta_1 + \eta_2$, and $p = p_1 + p_2$, we note that 
$$
\begin{aligned}
& e^{-\lambda_{p_1}(s) \langle k_1,\eta_1,p_1, L\eta_1 - K k_1 s \rangle} 
e^{-\lambda_{p_2}(s) \langle k_2,\eta_2,p_2, \eta' + L\eta_1  - Kk_1 s\rangle} 
\\&
\le 
C_{k,\eta,p,\eta',0}(s,t) C_{k,\eta,p,\eta',1}(s,t)C_{k,\eta,p,\eta',2}(s,t)e^{-\lambda_p(t) \langle k,\eta,p, \eta'  \rangle} 
\end{aligned}$$
where the factors $C_{k,\eta,p,j}(s,t)$ are defined by 
\begin{equation}\label{def-CCC}
\begin{aligned}
C_{k,\eta,p,\eta',0}(s,t) &: = e^{-(s^{-\delta} - t^{-\delta}) \langle k,\eta,p, \eta' \rangle}
\\
C_{k,\eta,p,\eta',1}(s,t) &: = e^{-(p_1^{-\delta} - p^{-\delta}) 
	\langle k_1,\eta_1,p_1, L \eta_1 - K k_1 s \rangle} 
\\C_{k,\eta,p,\eta',2}(s,t) &: = e^{-(p_2^{-\delta} - p^{-\delta})\langle k_2,\eta_2,p_2, \eta' + L\eta_1  - Kk_1 s \rangle}
\end{aligned}\end{equation}
each of which is smaller than one. These factors may be seen as gains coming from the monotonicity of $\lambda_p(t)$. Combining and noting $|k| \le 2 \langle k_1 \rangle \langle k_2\rangle$, we thus obtain 
\begin{equation}\label{est-source-eta}
\begin{aligned}
|\FS_{k,\eta,p}(t,\eta')| &\le C_0^pe^{-\lambda_p(t) \langle k,\eta,p, \eta' \rangle}\langle K k\rangle^{-1} \sum_{A_{k,\eta,p}}  \int_0^t  |\eta' + L\eta - Kks|
\\&\quad \times  C_{k,\eta,p,\eta',0}(s,t) C_{k,\eta,p,\eta',1}(s,t)C_{k,\eta,p,\eta',2}(s,t) \langle s \rangle^{-\sigma} \; ds. 
\end{aligned}\end{equation}
Evaluating at $\eta' = Kkt-L\eta$, we get 
\begin{equation}\label{est-source}
\begin{aligned}
|\FS_{k,\eta,p}(t,Kkt-L\eta)| &\le C_0^pe^{-\lambda_p(t) \langle k,\eta,p, L\eta -  K kt \rangle}  \sum_{A_{k,\eta,p}}  \int_0^t  (t-s)
\\&\quad \times  C_{k,\eta,p,0} C_{k,\eta,p,1} C_{k,\eta,p,2}(s,t) \langle s \rangle^{-\sigma} \; ds
\end{aligned}\end{equation}
with $C_{k,\eta,p,j}= C_{k,\eta,p,\eta',j}(s,t)$ for $\eta' = Kkt-L\eta$. The Lemma thus follows from the following claim
\begin{equation}\label{claimC}
\begin{aligned}
\sum_{A_{k,\eta,p}} \int_{0}^t (t-s) C_{k,\eta,p,0} C_{k,\eta,p,1} C_{k,\eta,p,2} (s,t) \langle s\rangle^{-\sigma}\; ds \le C_0 \langle t\rangle^{-\sigma}.
\end{aligned}
\end{equation}
Let us first bound the factors $C_{k,\eta,p,j}(s,t)$. 

\begin{lem} \label{lemC12}
Setting $C_{k,\eta,p,j}= C_{k,\eta,p,\eta',j}(s,t)$ as in \eqref{def-CCC} for $\eta' = Kkt-L\eta$, we have
$$
\begin{aligned}
C_{k,\eta,p,0}(s,t) &\le e^{-\theta_0(s^{-\delta} - t^{-\delta}) \langle k,\eta,p, kt\rangle}
\\
C_{k,\eta,p,1}(s,t) &\le  e^{-\theta_0(p_1^{-\delta} - p^{-\delta}) \langle k_1,\eta_1,p_1,  k_1 s \rangle} 
\\
C_{k,\eta,p,2}(s,t) &\le  e^{-\theta_0(p_2^{-\delta} - p^{-\delta})\langle k_2,\eta_2,p_2, kt - k_1 s\rangle} 
\end{aligned}$$ for some positive constant $\theta_0$. 
\end{lem}

\begin{proof}
Recalling the inequality \eqref{tri-kt}: 
$ 
| kt | \le |K kt - L \eta | +  |\eta|  ,
$
we have 
$$
\begin{aligned}
C_{k,\eta,p,1}(s,t) &= e^{-(s^{-\delta} - t^{-\delta}) \langle k,\eta,p, L \eta - Kkt  \rangle} 
\\&\le e^{-(s^{-\delta} - t^{-\delta}) \langle k,\eta  ,p \rangle/2} 
e^{-(s^{-\delta} - t^{-\delta}) \langle \eta, L \eta - Kkt  \rangle /2} 
\\&\le  e^{-\theta_0 (s^{-\delta} - t^{-\delta}) \langle k,\eta,p, kt \rangle} 
\end{aligned}$$
provided $\theta_0$ is small enough.
The bounds on $C_{k,\eta,p,1}(s,t) $ and $C_{k,\eta,p,2}(s,t) $ are similar. 
\end{proof}

Let us now prove the claim (\ref{claimC}).
To estimate the time integral, we consider two cases: $p_1\le p/2$ and $p_2\le p/2$. 

~\\
{\bf Case 1: $p_1\le p/2$.} In this case, we note that 
$$ 
p_1^{-\delta} - p^{-\delta} \ge \theta_\delta p_1^{-\delta} 
$$
for some positive constant $\theta_\delta$. This and the estimate from Lemma \ref{lemC12} yield
$$
C_{k,\eta,p,1}(s,t) \le  e^{-\theta_\delta p_1^{-\delta} \langle k_1,\eta_1,p_1,  k_1 s \rangle} .
$$
Let us further bound the exponent. Using the standard Young's inequality  
$ab \lesssim a^q + b^{q'}$, with $q = 1/(1-\delta)$ and $q' = q/(q-1)$, we have 
\beq \label{inconv}
|a|^{1-\delta}  = (a p_1^{-\delta})^{1-\delta} |p_1|^{\delta (1-\delta)} 
\le C_\delta \Big( |p_1|^{-\delta} |a| + |p_1|^{1-\delta} \Big)
\eeq
for some constant $C_\delta$. Using this with $a = \langle k_1,\eta_1,k_1 s \rangle$, we have 
$$
p_1^{-\delta} \langle k_1,\eta_1,p_1,  k_1 s \rangle \ge |p_1|^{1-\delta} 
+ |p_1|^{-\delta} \langle k_1,\eta_1,k_1 s \rangle \ge \frac{1}{C_\delta} \langle k_1,\eta_1,k_1 s \rangle^{1-\delta}.
$$
Clearly, we also have 
$ p_1^{-\delta} \langle k_1,\eta_1,p_1,  k_1 s \rangle \ge \langle p_1\rangle^{1-\delta} $, recalling $p_1 \in \NN^*$. This yields 
\beq \label{inconv2} 
p_1^{-\delta} \langle k_1,\eta_1,p_1,  k_1 s \rangle \ge \frac{1}{2C_\delta} \langle k_1,\eta_1,p_1,k_1 s \rangle^{1-\delta}.
\eeq
Therefore, 
\begin{equation}\label{est-Cp222}
C_{k,\eta,p,1}(s,t) \le  e^{-\theta_\delta p_1^{-\delta} \langle k_1,\eta_1,p_1,  k_1 s \rangle} \le   
e^{-\theta'_\delta \langle k_1,\eta_1,p_1 \rangle^{1-\delta}} 
e^{-\theta'_\delta \langle k_1 s \rangle^{1-\delta}} , 
\end{equation}
for some positive constant $\theta'_\delta$. 

On the other hand, we simply bound 
$$
C_{k,\eta,p,0}(s,t) \le e^{-\theta_0(s^{-\delta} - t^{-\delta}) \langle k,\eta,p, kt\rangle} 
\le e^{-\theta_0(s^{-\delta} - t^{-\delta}) \langle t\rangle} 
$$
noting $k \not =0$. We also bound $C_{k,\eta,p,2}(s,t) \le 1$. Inserting these estimates into \eqref{claimC}, we have 
$$
\begin{aligned}
&
\sum_{A_{k,\eta,p}} \int_{0}^t  (t-s) C_{k,\eta,p,0}(s,t) C_{k,\eta,p,1}(s,t) 
C_{k,\eta,p,2}(s,t) \langle s\rangle^{-\sigma}\; ds 
\\&\le \sum_{A_{k,\eta,p}} e^{-\theta'_\delta \langle k_1,\eta_1,p_1 \rangle^{1-\delta}} \int_{0}^t  (t-s) 
e^{-\theta_0(s^{-\delta} - t^{-\delta}) \langle t\rangle} 
e^{-\theta'_\delta \langle k_1 s \rangle^{1-\delta}}   \langle s\rangle^{-\sigma}\; ds 
\\&\lesssim \int_{0}^t  (t-s) 
e^{-\theta_0(s^{-\delta} - t^{-\delta}) \langle t\rangle} e^{-\theta'_\delta \langle s \rangle^{1-\delta}}  
 \langle s\rangle^{-\sigma}\; ds ,\end{aligned}
$$
in which we used $e^{-\theta'_\delta \langle k_1 s \rangle^{1-\delta}}  \le e^{-\theta'_\delta \langle s \rangle^{1-\delta}} $, 
since $k_1\not =0$. 
It remains to bound the time integral
$$
\int_{0}^t  (t-s) 
e^{-\theta_0(s^{-\delta} - t^{-\delta}) \langle t\rangle} e^{-\theta'_\delta \langle s \rangle^{1-\delta}}   
\langle s\rangle^{-\sigma}\; ds
\le C_0 \langle t\rangle^{-\sigma} .
$$ 
Indeed, the estimate is clear for $s\ge t/2$, using the exponential term 
$e^{-\theta'_\delta \langle s\rangle^{1-\delta}}$ in the integrand. 
On the other hand, for $s\le t/2$, we make use of the fact that  $s^{-\delta} - t^{-\delta} \ge \theta_\delta t^{-\delta}$, 
yielding again an exponential decaying term 
$$
e^{-(s^{-\delta} - t^{-\delta}) \langle t\rangle} \le  e^{-\theta_\delta \langle t\rangle^{1-\delta}} .$$
The claim \eqref{claimC} follows.

~\\
{\bf Case 2: $p_2\le p/2$.} Similarly, in this case, we use 
$$ p_2^{-\delta} - p^{-\delta} \ge \theta_\delta p_2^{-\delta} $$
for some positive constant $\theta_\delta$, which implies 
$$
C_{k,\eta,p,2}(s,t) \le  e^{-\theta_\delta p_2^{-\delta} \langle k_2,\eta_2,p_2, kt -  k_1 s \rangle} .
$$
Estimating the exponent exactly as done in \eqref{est-Cp222}, we thus obtain 
\begin{equation}\label{bd-CC2}
C_{k,\eta,p,2}(s,t) \le   e^{-\theta_\delta \langle k_2,\eta_2,p_2  \rangle^{1-\delta}}   
e^{-\theta_\delta \langle  kt - k_1 s \rangle^{1-\delta}} .
\end{equation}
In the case when $|kt - k_1 s|\ge t/2$, the above yields an exponential decay term in $(k_2,\eta_2,p_2,t)$. 
The claim \eqref{claimC} thus follows. 

It remains to consider the case when $|kt - k_1 s|\le t/2$.  
It suffices to treat the case $k> 0$, the other being similar. 
In this case, we note that $k_1>0$ and $s\in [k_1^{-1} (k-1/2)t , k_1^{-1}(k+1/2) t] $. 
In particular, as $s<t$, we have $k_1\ge k$. We treat two cases $k_1 = k$ and $k_1>k$, separately.

Consider first the case when $k_1 = k \not =0$. We then have 
$$
C_{k,\eta,p,2}(s,t) \le  e^{-\theta_\delta \langle k_2,\eta_2,p_2,  k(t-s) \rangle^{1-\delta}} 
\le  e^{-\theta_\delta \langle k_2,\eta_2,p_2\rangle^{1-\delta}} e^{-\theta_\delta \langle t-s\rangle^{1-\delta}},
$$
while we simply bound $C_{k,\eta,p,0}(s,t) \le 1$ and $C_{k,\eta,p,1}(s,t) \le 1$. Let us now check the claim \eqref{claimC} for this case. We have 
$$
\begin{aligned}
& 
\sum_{A_{k,\eta,p}} \int_{0}^t (t-s) C_{k,\eta,p,0}(s,t)C_{k,\eta,p,1}(s,t) 
C_{k,\eta,p,2}(s,t) \langle s\rangle^{-\sigma}\; ds
\\& \le \sum_{A_{k,\eta,p}} e^{-\theta_\delta \langle k_2,\eta_2,p_2\rangle^{1-\delta}} 
\int_{0}^t (t-s) e^{-\theta_\delta \langle t-s\rangle^{1-\delta}} \langle s\rangle^{-\sigma}\; ds
\\&\lesssim \int_{0}^t e^{-\frac12 \theta_\delta \langle t-s\rangle^{1-\delta}} \langle s\rangle^{-\sigma}\; ds
,\end{aligned}
$$
which is clearly bounded by $C_0 \langle t\rangle^{-\sigma}$.
 
Next, we consider the case when $k_1>k>0$. In this case, recalling \eqref{bd-CC2}, we have
$$
C_{k,\eta,p,2}(s,t) \le   e^{-\theta_\delta \langle k_2,\eta_2,p_2  \rangle^{1-\delta}} ,
$$
while we use the following bound on $C_{k,\eta,p,0}(s,t)$: 
$$C_{k,\eta,p,0}(s,t) \le e^{-\theta_0(s^{-\delta} - t^{-\delta}) \langle kt\rangle} .$$
Since $s\in [k_1^{-1} (k-1/2)t , k_1^{-1}(k+1/2) t] $ and $k_1>k>0$, we bound 
$$
s^{-\delta} - t^{-\delta} \ge  \Big( \frac{k_1^\delta}{(k+1/2)^\delta} - 1 \Big) \frac{1}{t^\delta} \ge \theta_\delta t^{-\delta}|k|^{-1}
$$ 
for some positive constant $\theta_\delta$ independent of $k,k_1$. 
This proves 
$$
C_{k,\eta,p,0}(s,t)\le e^{-\theta_\delta \langle t\rangle^{1-\delta}}. 
$$
We also bound $C_{k,\eta,p,1}(s,t)\le 1$. Combing the estimates into \eqref{claimC}, we thus have 
$$
\begin{aligned}
& 
\sum_{A_{k,\eta,p}} \int_{0}^t (t-s) C_{k,\eta,p,0}(s,t)C_{k,\eta,p,1}(s,t) 
C_{k,\eta,p,2}(s,t) \langle s\rangle^{-\sigma}\; ds
\\& \le \sum_{A_{k,\eta,p}} e^{-\theta_\delta \langle k_2,\eta_2,p_2\rangle^{1-\delta}} 
\int_{0}^t (t-s) e^{-\theta_\delta \langle t\rangle^{1-\delta}} \langle s\rangle^{-\sigma}\; ds
\end{aligned}
$$
which is again bounded by $C_0 \langle t\rangle^{-\sigma}$. The claim \eqref{claimC} follows. 
\end{proof}

\subsection{Estimates on $\Ff_{k,\eta,p}$}

In this section, we prove the estimates \eqref{est-Fp} on $\Ff_{k,\eta,p}$:
\begin{equation}
\label{re-estf}
|\Ff_{k,\eta,p} (t,\eta')| \le C_0^p  e^{-\lambda_p(t) \langle k,\eta,p, \eta' \rangle} \langle k\rangle^{-1}
\end{equation}
assuming that the estimates \eqref{est-Fp} on $\Ff_{k,\eta,p_1}$ hold for all $p_1\le p-1$ and the estimates \eqref{est-Ep} on and $\FE_{k,\eta,p_1}$ hold for all $p_1\le p$. This will end the proof of Proposition \ref{prop-fkep}. By Lemma \ref{lem-linearFkp}, we have 
$$
\Ff_{k,\eta,p}(t,\eta') = \FS_{k,\eta,p}(t,\eta') - \int_0^t \FE_{k,\eta,p}(s) \widehat{\partial_v \mu}(\eta'+ L\eta - Kk s) \; ds .$$
Using \eqref{est-Ep} and the analyticity assumption on $\mu(v)$, we get 
$$
\begin{aligned}
&\int_0^t |\FE_{k,\eta,p}(s) \widehat{\partial_v \mu}(\eta'+ L\eta - Kk s)| \; ds 
\\& \le C_0^{p} \langle Kk\rangle^{-1}\int_0^t e^{-\lambda_{p}(s) \langle  k,\eta,p,L\eta - K k s\rangle} e^{-\theta_0 | \eta' + L\eta - Kk s |}\langle s \rangle^{-\sigma} \; ds .
\\& \le C_0^{p} \langle Kk\rangle^{-1} e^{-\lambda_{p}(t) \langle  k,\eta,p,\eta' \rangle} \int_0^t \langle s \rangle^{-\sigma} \; ds
\\& \le C_0^{p} \langle k\rangle^{-1} e^{-\lambda_{p}(t) \langle  k,\eta,p,\eta'  \rangle}
\end{aligned}$$ 
in which we used that $\lambda(t) \le \lambda(s)\le \theta_0$.

It remains to give estimates on $\FS_{k,\eta,p}(t,\eta')$. Recall from \eqref{est-source-eta} that 
$$\begin{aligned}
|\FS_{k,\eta,p}(t,\eta')| &\le C_0^pe^{-\lambda_p(t) \langle k,\eta,p, \eta'  \rangle}\langle K k\rangle^{-1} \sum_{A_{k,\eta,p}}  \int_0^t  |\eta' + L\eta - Kk s|
\\&\quad \times  C_{k,\eta,p,\eta',0}(s,t) C_{k,\eta,p,\eta',1}(s,t)C_{k,\eta,p,\eta',2}(s,t) \langle s \rangle^{-\sigma} \; ds,
\end{aligned}
$$
where the factors $C_{k,\eta,p,\eta',j}(s,t)$ are defined as in \eqref{def-CCC}. Since $K\ge 1$ and $L\lesssim K$, we have 
$$K^{-1}| \eta' + L \eta   - K k s)| \le |\eta' | +  |\eta| + |ks| \lesssim \langle s\rangle \langle k,\eta,p,\eta' \rangle. $$

The claim \eqref{re-estf} will follow from the following estimates, which we will now prove
\begin{equation}\label{claim-sumC} \sum_{A_{k,\eta,p}} C_{k,\eta,p,\eta',1}(s,t)C_{k,\eta,p,\eta',2}(s,t) \lesssim 1,
\end{equation}
and 
\begin{equation}\label{time-shrink}
\int_{0}^t e^{-(s^{-\delta} - t^{-\delta})\langle k,\eta,p, \eta' \rangle} \langle k,\eta,p, \eta' \rangle
\langle s\rangle^{-\sigma+1} \; ds \lesssim 1, 
\end{equation}
uniformly in $k,\eta,p, \eta'$, and $t$. 

Let us start with \eqref{claim-sumC}. As argued above, we have, for $p_1\le p/2$, 
$$ 
\begin{aligned}
C_{k,\eta,p,\eta',1}(s,t)
&\le e^{-\theta_\delta p_1^{-\delta}\langle k_1,\eta_1,p_1, L\eta_1 - K k_1 s \rangle} 
\le e^{-\theta_\delta \langle k_1,\eta_1,p_1\rangle^{1-\delta}},
\end{aligned}
$$
Similarly, for $p_2\le p/2$, we have 
$$
\begin{aligned}
C_{k,\eta,p,\eta',2}(s,t) &\le e^{-\theta_\delta p_2^{-\delta}\langle k_2,\eta_2,p_2,\eta' - L\eta_1 + K k_1 s\rangle}
\le e^{-\theta_\delta \langle k_2,\eta_2,p_2\rangle^{1-\delta}} .
\end{aligned}
$$
In both cases, the claim \eqref{claim-sumC} follows in view of the definition of $A_{k,\eta,p}$. 
Finally, we check \eqref{time-shrink}. 
We have $\langle s\rangle^{-\sigma + 1}\lesssim |\frac d{ds} s^{-\delta}|$. Therefore, 
$$
\begin{aligned}
&\int_{0}^te^{-(s^{-\delta} - t^{-\delta}) 
\langle k,\eta,p, \eta' \rangle}\langle k,\eta,p, \eta' \rangle\langle s\rangle^{-\sigma+1} \; ds
\\
&\lesssim \int_{0}^t e^{-(s^{-\delta} - t^{-\delta}) \langle k,\eta,p, \eta' \rangle}  \langle k,\eta,p,\eta' \rangle |\frac d{ds} s^{-\delta}| \; ds 
\end{aligned}
$$
which is bounded. This ends the proof of Proposition \ref{prop-fkep}. 

%\end{proof}


\begin{thebibliography}{99}


\bibitem{Bedrossian} J. Bedrossian, Nonlinear echoes and Landau damping with insufficient regularity, 
{\em arXiv:1605.06841.} 2016. 

\bibitem{BMM} J. Bedrossian, N. Masmoudi, and C. Mouhot,
Landau damping: paraproducts and Gevrey regularity. 
{\em Ann. PDE} 2 (2016), no.1, Art.4, 71pp. 

\bibitem{Gould} 
R. W. Gould, T. M. O'Neil, and J. H. Malmberg, 
Plasma Wave Echo. 
{\em Phys. Rev. Lett.} 19 (1967), no 5. 


\bibitem{GNR1} 
E. Grenier, T. T. Nguyen, and I. Rodnianski.
\newblock Landau damping for analytic and Gevrey data. 


\bibitem{MV} C. Mouhot and C. Villani, On Landau damping.
{\em Acta Math.} 207 (2011), no. 1, 29-201. 

\bibitem{Landau} L. Landau, On the vibrations of the electronic plasma. (Russian)
{\em Akad. Nauk SSSR. Zhurnal Eksper. Teoret. Fiz.} 16, (1946). 574-586. 

\bibitem{Lin} Z. Lin and C. Zeng, Small BGK waves and nonlinear Landau damping. 
{\em Comm. Math. Phys.} 306 (2011), no. 2, 291-331. 


\end{thebibliography}
\end{document}